\documentclass[11pt, reqno]{amsart}% cls. for J. Nonlinear Sci. Appl

\usepackage{amsmath, amsthm, amscd, amsfonts, amssymb, graphicx, color}
\usepackage{lmodern}
\RequirePackage[backref,colorlinks=true]{hyperref}
\begin{document}

\title[Approximation of continuous interval-valued functions]{Constructive approximation of continuous interval-valued functions}

%    Remove any unused author tags.

%    author one information
\author{Juan J. Font}
\address{Institut Universitari
de Matem\`{a}tiques i Aplicacions de Castell\'{o} (IMAC),
Universitat Jaume I, Avda. Sos Baynat s/n, 12071, Castell\'{o}
(Spain).}
\thanks{Research supported by grant PID2019-106529GB-I00 and grant
UJI-B2019-08}
%\curraddr{}
\email{font@uji.es}

%    author two information
\author{Sergio Macario}

\address{Institut Universitari
de Matem\`{a}tiques i Aplicacions de Castell\'{o} (IMAC),
Universitat Jaume I, Avda. Sos Baynat s/n, 12071, Castell\'{o}
(Spain).}
%\thanks{The second author is supported by grant  UJI-B2019-08}
%\curraddr{}
\email{macario@uji.es}



\date{\today}

% {\tt
%\begin{center}Communicated by ??? \end{center}}}
%

%The name will be entered by editor

%\cortext[c1]{Corresponding author}
%\setcounter{page}{1}
%\vol{? (201?)} \pages{1--?}
%% The correct information will be entered by the editor
%
%\recivedat{received date}
%% The correct dates will be entered by the editor

%\authors{J.J. Font, S. Macario}

%\doi{\href{doi.org/10.22436}{10.22436/jnsa.010.01.01}}
% The correct doi will be entered by the editor

\begin{abstract}

In this paper we present a Stone-Weierstrass type result in the context of continuous
interval-valued functions
defined on a compact Hausdorff space. Namely, we provide a constructive proof of the approximation.

%We also apply such results to the study of the inverse shadowing property in the realm of interval-valued dynamical systems.
\end{abstract}

\maketitle

%Theorem-like structures
%If you need new environments, define them here with the command \newtheorem{envname}{caption}.

%Some environments such as definitions, theorems, lemmas or examples, have defined  in the list below.

\newtheorem{theorem}{Theorem}[section]
\newtheorem{lemma}[theorem]{Lemma}
\newtheorem{proposition}[theorem]{Proposition}
\newtheorem{corollary}[theorem]{Corollary}
\newtheorem{question}[theorem]{Question}

\theoremstyle{definition}
\newtheorem{definition}[theorem]{Definition}
\newtheorem{algorithm}[theorem]{Algorithm}
\newtheorem{conclusion}[theorem]{Conclusion}
\newtheorem{problem}[theorem]{Problem}

\theoremstyle{remark}
\newtheorem{remark}[theorem]{Remark}
\numberwithin{equation}{section}

\section{Introduction}

When modelling a complex process, some uncertainty in the input and
output variables is expected. In such situations, input and output data
should mix  numbers and intervals. So, often
in these modelling processes, it is needed to deal with functions giving an
interval as an output rather than real numbers.
Interval Analysis is a relatively new area of mathematics which studies how to handle such interval uncertainty which appears in a lot of computer-mathematical models of certain real-world phenomena; in particular, in control theory, linear programming, optimization problems, etc.
The first principles of interval arithmetic were set independently, and almost simultaneously, in the fifties by Paul S. Dwyer (\cite{dwyer}) and Ramon E. Moore (\cite{moore1}, \cite{moore2}) in the Unites States, Mieczyslaw Warmus (\cite{warmus}) in Poland, and Teruo Sunaga (\cite{sunaga}) in Japan, although Interval Analysis is often said to have begun with Moore's book
\cite{moore:66}.

%Interval Analysis plays a key role in Fuzzy Analysis since the level sets of fuzzy numbers are intervals of real numbers.

 Interval-valued functions, that is, functions defined on a topological space taking values in the space of closed intervals, should play a central role in Interval Analysis, just like real-valued functions do in the Classical Analysis.
 % Interval-valued functions have become the main tool in several interval  contexts.
% , such as fuzzy differential equations (\cite{BG}), fuzzy integrals (\cite{WG}) or fuzzy optimization (\cite{HGL}).
However some difficulties arise when dealing with these interval-valued functions, mainly because the space they form is not a linear space; indeed it is not a group with respect to addition.

Stimulated by the interaction between Interval Analysis and Optimization Theory, interval-valued functions have received considerable attention recently motivated mainly by the necessity of formulating a formal framework for a differential calculus in this context. Namely, an appropriate metric space and a well-behaved subtraction are needed for such development. Thus, for the theoretical framework of calculus of interval-valued functions and among others, Hukuhara (\cite{huku:67}) introduced a concept of difference of two intervals (H-difference) which was used to define the $H$-derivative of an interval-valued function. This concept, however, turned out to be a very restrictive concept. So, in 2008, Stefanini proposed an improved version of the $H$-difference (the $gH$-difference) which seems to be a very useful tool for dealing with interval-valued functions.

The literature on other aspects of the theory of interval-valued functions is reduced compared to the plethora of results dealing with the differential calculus mentioned in the previous paragraph. One of such aspects is the study of the approximation of continuous interval-valued functions.

 In this paper, by combining certain techniques from \cite{FSS:17}, \cite{Je}) and \cite{prola}, we provide some sufficient conditions on a subset of  the space of continuous interval-valued  functions in order that it be dense, which is to say, a Stone-Weierstrass type result for interval-valued continuous functions which doesn't seem to have made its way into the literature yet.
The proofs are constructive and use only straightforward concepts. 

In this context, we also provide a Jackson type approximation result involving the modulus of continuity of interval-valued functions based on the $gH$-difference mentioned above.

Finally, based on the results of the previous sections and taking advantage of an striking result by Guliyev and Ismailov (\cite{GI}), we show how an interval-valued continuous function can be approximated using interval neural networks.

\section{Preliminaries}

%A basic result in fuzzy number theory is the so-called
%representation theorem of fuzzy numbers on $\mathbb{R}$, which was
%provided by Goetschel and Voxman in \cite{GV}:

Following \cite{DK2}, let $\mathcal{K_{C}}$ denote the set of all finite closed intervals of the real line $\mathbb{R}$.
$$\mathcal{K_{C}}=\{[a,b]\ : a,b\in\mathbb{R},\ a\leq b\}.$$
where $[a,a]$ denotes the singleton $\{a\}$.

In $\mathcal{K_{C}}$ we shall consider two operations, (Minkowski) addition and scalar multiplication, defined by
$$I+J=\{a+b\ :\ a\in I,\ b\in J\}$$
and, for every $\lambda\in\mathbb{R}$,
$$\lambda I=\{\lambda a\ : a\in I\}.$$

 Hereafter $0$ will denote the singleton $[0,0]=\{0\}$. and it is the neutral element for the Minkowski addition. Unless
 $J$ is a singleton , $J+(-J)\neq0$. So,  in general, there is no inverse for the sum operation and
the structure of $\mathcal{K_{C}}$ is that of a cone rather than a linear space.

We consider now the metric space $(\mathcal{K_{C}},d_H)$, where $d_H$ is the Pompeiu-Hausdorff metric defined as

$$d_H(I,J)=\max\left\{\max_{a\in I}d(a,J),\max_{b\in J}d(I,b)\right\}$$
where
$$d(a,J)=\min_{b\in J} |a-b|\quad \text{and} \quad d(I,b)=\min_{a\in I} |a-b|.$$

In particular, if $I=[a,b]$ and $J=[c,d]$, then
$$d_H(I,J)=\max\{|a-c|,|b-d|\}.$$

It is well known that $(\mathcal{K_{C}},d_H)$ is a complete metric space (see \cite{aubin-franskowska:1990}).

\begin{proposition} \label{properties} The metric $d_H$ satisfies the following properties:
\begin{enumerate}
\item $d_H ( \sum_{i=1}^{m}A_{i}, \sum_{i=1}^{m}B_{i}) \le \sum_{i=1}^{m}d_H (A_i, B_i)$ where $A_i,B_i\in \mathcal{K_{C}}$ for $i=1,...,m$.
\item  $d_H( \alpha A, \alpha B) = \alpha d_H (A, B)$ where $A,B\in \mathcal{K_{C}}$ and $\alpha>0$.
\item $d_H (\alpha A,\beta A ) = \mid \alpha - \beta  \mid d_H (A,0),$ where $A \in \mathcal{K_{C}}$,  $ \alpha,\, \beta \ge 0$.
\item $d_H (\alpha A,\beta B ) \le \mid \alpha - \beta  \mid d_H (A,0) + \beta d_H ( A,B),$ where $A,B \in \mathcal{K_{C}}$, $ \alpha,\; \beta \ge 0$.
\item $d_H(A+C,B+C)=d_H(A, B)$, where $A, B, C\in \mathcal{K_C}$.
\end{enumerate}
\end{proposition}

\begin{proof} The proofs of (1) and (2) can be found, for example, in \cite{DK2}.
In order to prove (3),  let us assume that $\beta < \alpha$ and rewrite $\alpha A$ and  $\beta A$ as
$\alpha A =  (\beta + (\alpha- \beta)) A=\beta A + (\alpha- \beta) A $ and $ \beta A= \beta A + (\alpha- \beta) 0$.
By (1), we know that
$$d_H(\alpha A, \beta A) \le d_H (\beta A, \beta A) + d_H ((\alpha - \beta) A ,(\alpha - \beta) 0) =
\mid  \alpha - \beta \mid  d_H (A, 0).$$

\bigskip
Consequently, by (2) and (3),
$$d_H (\alpha A,\beta B) \le d_H (\alpha A, \beta A) + d_H ( \beta A, \beta B) \le \mid \alpha - \beta  \mid d_H (A,0) + \beta d_H ( A,B).$$
%we have, by (3),
%$$ d_\infty (ku, \mu u) \le \mid k - \mu \mid d_\infty (u, 0) .$$
%Furthermore, by (2), we can write
%$$ d_\infty ( \mu u, \mu v) = \mu d_\infty (u, v) .$$
%Hence, clearly we infer
%$$d_\infty (ku,\mu v ) \le \mid k - \mu  \mid d_\infty (u,0) + \mu d_\infty ( u,v). $$
For (5), take $A=[a^-,a^+]$,  $B=[b^-,b^+]$ and $C=[c^-,c^+]$. Then
\[
\begin{split}
d_H(A+C,B+C)&=\max\{|a^-+c^- - (b^- + c^-)|, \ |a^++c^+-(b^++c^+|\} \\
&=\max\{|a^- - b^- |, \ |a^+ - b^+|\}=d_H(A,B)
\end{split}
\]
\end{proof}

\bigskip

In the sequel, let $K$ be a compact Hausdorff space and let $C(K,\mathcal{K_{C}})$ denote the space of continuous functions from $K$ to the metric space $(\mathcal{K_{C}}, d_H)$, that is, the space of continuous interval-valued functions defined on $K$. We shall consider $C(K,\mathcal{K_{C}})$ endowed with the supremum metric:
$$D_{\infty}(f,g) =  \sup_{ t \in K} d_H(f(t),g(t)),$$
which induces, as usual, the uniform convergence topology on $C(K,\mathcal{K_{C}})$.

\begin{proposition} Let $f \in C (K, \mathcal{K_{C}})$ and $\varphi \in C (K, \mathbb{R}^+)$. Then the mapping from $K$ to $\mathcal{K_{C}}$ defined by $k \mapsto \varphi (k) f(k)$, belongs to
$C(K, \mathcal{K_{C}})$ as well.
\end{proposition}

\begin{proof}
Since $\varphi(k)>0$ for all $k\in K$, then  the pointwise product $\varphi (k) f(k)$ is well defined.  Let us see that the mapping  is continuous. Choose $k_1,k_2 \in K$. By Proposition \ref{properties} (4), we know that
\[
\begin{split}
d_H( \varphi (k_1) f(k_1), \varphi (k_2) f(k_2) ) &\le | \varphi (k_1) - \varphi (k_2) |  d_H ( f(k_1),0)\\
 &+ \varphi (k_2) d_H ( f(k_1), f(k_2)).
 \end{split}
 \]
The proof follows since both $\varphi$ and $f$ are continuous and bearing in mind that every continuous function from a compact space to a metric space is bounded.
\end{proof}

\section{A version of the Stone-Weierstrass theorem in Interval Valued Analysis.}
The Stone-Weierstrass Theorem is a main tool in the General Approximation Theory. Our version states conditions for a subset $H$ of continuous interval-valued  functions from a compact space $K$  being dense in the whole space of continuous interval-valued functions, endowed with the supremum metric. That is, to find, for every continuous interval-valued function $f$, another one close enough to $f$ and belonging to this smaller subset of continuous interval-valued functions.
Following \cite{FSS:17}, we introduce a useful tool to get  our main theorem (Theorem \ref{main}).

\begin{definition} \label{multi} Let $H$ be a nonempty subset of $C(K,\mathcal{K_{C}})$. We define $$Conv(H)=\{ \varphi \in C(K, [0,1]): \varphi f + (1 - \varphi) g \in H \hspace{0,2cm} for \hspace{0,2cm} all \hspace{0,2cm} f,g \in H \}.$$
\end{definition}

This set, $Conv(H)$, have some nice properties that will be helpful in order to prove our results.

\begin{proposition} \label{multiplier} Let $H$ be a nonempty subset of $C(K,\mathcal{K_{C}})$. Then we have:
\begin{enumerate}
\item $\phi \in Conv(H)$ implies that $1 - \phi \in Conv(H)$.
\item If $\phi, \varphi \in Conv(H)$, then $\phi \cdot \varphi \in Conv(H)$.
\item If $\phi$ belongs to the uniform closure of $Conv(H)$, then so does $1 - \phi$.
\item If $\phi, \varphi$ belong to the uniform closure of $Conv(H)$, then so does $\phi \cdot \varphi$.
\end{enumerate}
\end{proposition}

\begin{proof}
 (1) is clear. To see (2), let us assume that $\phi,\varphi \in Conv(H)$.  To get that  $\phi \cdot \varphi \in Conv(H)$, we use the identity
$$1 - \phi \cdot \varphi = (1 - \phi) + \phi (1 - \varphi)$$
which implies, for every pair $f, g \in H$, that
$$(\phi \cdot \varphi)f + (1 - \phi \cdot \varphi)g = \phi [ \varphi f + (1 - \varphi)g ] + (1 - \phi) g\in H.$$

To show (3), let us suppose that there exists a sequence $\{\phi_n \}\subset Conv(H)$ uniformly converging to $\phi\in C(K,[0,1])$. Hence, $\{1 - \phi_n\}$, which is contained in $Conv(H)$ by (1), converges uniformly to $1 - \phi$.

Finally, (4) can be proved in the same way.

\end{proof}

We still need to state another property of  $Conv(H)$ whose proof follows from two technical lemmas that can be found in \cite{Je}.

%Next we state two technical lemmas which will used in the sequel.

\begin{lemma}\cite[Lemma 2]{Je}\label{patata} Let $0\leq a < b \leq 1$ and $0 < \delta < \frac{1}{2}$. There exists a polynomial $p(x)= (1 - x^m)^n$ such that
\begin{enumerate}
\item $p(x)> 1 - \delta$ for all $0 \le x \le a$,
\item $p(x)< \delta$ for all $b \le x \le 1$.
\end{enumerate}
\end{lemma}

%The next lemma can be deduced from  \cite[Theorem 1]{Je}.

\begin{lemma}\cite[Theorem 1]{Je}\label{patata2} For every $0<\varepsilon<\frac{1}{32}$, there exists a polynomial in two variables $q(x,y)= (1 -p_1(x))p_2(x)(1-p_3(y))p_4(y)$, with $p_k(z)=(1-z^{m_k})^{n_k}$, $m_k,\; n_k \in\mathbb{N}$,
$k=1,2,3,4$; such that
$$|x\wedge y-q(x,y)|<\varepsilon, \qquad \text{for every $x,y\in [0,1]$ }$$
where $x\wedge y$ stands for $min\{x,y\}$.
\end{lemma}

\begin{proposition} \label{patata3} Let $ H \subseteq C(K,\mathcal{K_{C}})$. Given two elements of $Conv(H)$, its maximum lies in the uniform closure of $Conv(H)$.
\end{proposition}
\begin{proof}
Let $\phi$ and $\psi$ be two elements in $Conv(H)$. Let $\phi\vee \psi$ stand for the maximum of $\phi$ and $\psi$. By property (3) of Proposition~\ref{multiplier}, since
$$\phi\vee \psi =1-((1-\phi)\wedge(1-\psi)),$$
 it suffices to prove that $(1-\phi)\wedge(1-\psi)$ belongs to the uniform closure of $Conv(H)$.

Take $0<\varepsilon<1/32$ and the corresponding polynomial $q(x,y)$ given by Lemma~\ref{patata2}. Then,
\[
\left| (1-\phi(s)\wedge(1-\psi(s)-q(1-\phi(s),1-\psi(s))\right|<\varepsilon
\]
for all $s\in K$. By (1) and (2) of Proposition~\ref{multiplier} and the given form of $q(x,y)$ we can claim that $\varphi:=q(1-\phi,1-\psi)$ belongs to $Conv(H)$. So we have just shown that there exists $\varphi\in Conv(H)$ with
\[
\left| (1-\phi(s))\wedge(1-\psi(s))-\varphi(s)\right|<\varepsilon
\]
for all $s\in K$ and we get the conclusion.

\end{proof}

It is convenient to remark that not only the maximum  of two elements of $Conv(H)$ belongs to the uniform closure of $Conv(H)$, but also its minimum.

%HASTA AQUI 17/3/20

%%%%%%%%%%%%%%%%%%%%%%%%%%%%%%%%%%%%
%Next definition is the classical one about continuous functions separating points.

\begin{definition}
Let $H$ be a subset $C(K, [0, 1])$.  It is said that $H$ separates the points of $K$ if given $s, t \in K$,
there exists $\phi \in H$ such that $\phi(s) \not= \phi(t)$.
\end{definition}

The next lemma provides the selection of the elements $\phi\in Conv(H)$ needed for the construction of the approximation function $g$ in our main result. The proof is similar to the one from \cite{FSS:17} but we include it here for the sake of completeness.

\begin{lemma}\label{pat3} Let $ H$ be a subset of $C(K,\mathcal{K_{C}})$ such that $Conv(H)$ separates the points of $K$. Given $x_0 \in K$ and an open neighborhood $V$ of $x_0$, there exists a neighborhood $U$ of $x_0$, with $U\subseteq V$, such that, for all $0 < \delta < \frac{1}{2}$, there is $\phi \in Conv(H)$ satisfying
\begin{enumerate}
\item $\phi (t) > 1 - \delta,$ for all $t \in U$;
\item $\phi (t) < \delta,$ for all $t \notin V$.
\end{enumerate}

\end{lemma}

\begin{proof}
Let $W=K\setminus V$. Since $Conv(H)$ separates the points of $K$, we can assume, with no loss of generality, that for each $t \in W$, there is a $\varphi_t \in Conv(H)$ such that $\varphi_t(t) < \varphi_t(x_0)$.

Pick two real numbers $a_t$ and $b_t$ such that $\varphi_t(t) < a_t < b_t < \varphi_t(x_0)$.
Taking $\delta = \frac{1}{4}$ in Lemma \ref{patata}, we can find a polynomial $p_t(x)= (1 -x^m)^n$ such that $p_t(x) < \frac{1}{4}$ for $b_t \le x \le 1$, and $p_t(x) > \frac{3}{4}$ for $0 \le x \le a_t$.
Hence, $p_t(\varphi_t(x_0)) < \frac{1}{4}$ and $p_t(\varphi_t(t)) > \frac{3}{4}$.

Then, for every $t\in W$, we can define
\[U(t):=\{s \in K : p_t(\varphi_t(s))> \frac{3}{4} \},\]
which is
an open neighborhood of  $t$.
Since $W$ is compact, there exist $t_1,...,t_m \in W$ such that $W\subset U(t_1)\cup U(t_2) \cup ... \cup U(t_m)$.
For each $i=1,...,m$ and all $s\in K$ we can  define
$$\varphi_i(s)=p_{t_i}(\varphi_{t_i}(s)).$$
We have $p_{t_i}(\varphi_{t_i}(s)) = ( 1 - [\varphi_{t_i}(s)]^m)^n$ and since $\varphi_{t_i}(s) \in Conv(H)$, we infer, by Proposition~\ref{multiplier},  that so is $\varphi_i=p_{t_i}(\varphi_{t_i})$,
for all $i=1,...,m$.

Let us define $\psi (s) = \varphi_1(s)\vee...\vee\varphi_m(s)$, $s \in K$ and, by Proposition~\ref{patata3},
we know that $\psi$ lies in the uniform closure of $Conv(H)$.
We remark that $\psi(x_0) < \frac{1}{4}$ and $\psi (t) > \frac{3}{4}$, for all $t \in W$ due to the properties of the polynomials $p_{t}(x)$. Now, let us define $$U = \{ s \in K; \psi (s) < \frac{1}{4} \}.$$
Clearly, $U$ is an open neighborhood of $x_0$ in $K$.
We claim that $U$ is contained in $V$. Indeed, if $s\in U$ and $s \notin V$, then $s \in W$ and, consequently,
$\psi (s) > \frac{3}{4}$ which means it cannot be in $U$.

 Take  $0<\delta< \frac{1}{2}$ and let $p$ be the polynomial defined by Lemma~\ref{patata}, applied to $a = \frac{1}{4}$, $b = \frac{3}{4}$ and $\delta/2$. Define $\mu(s)= p(\psi (s))$, for $s \in K$.
By Proposition~\ref{multiplier}, (3) and (4), the function $\mu$ also belongs  to the uniform closure of $Conv(H)$.

If $s \in U$, then $\mu(s) > 1 -\delta/2 $ by construction.
If $s \notin V$, then $s \in W$ and  $\psi (s) > \frac{3}{4}$ gives $\mu(s) < \delta/2$.

Since $\mu$ belongs to the uniform closure of $Conv(H)$,  there exists $\phi \in Conv(H)$ such that $ \parallel \phi - \eta \parallel _{\infty} = \sup_{s \in K} \mid \phi(s) - \mu(s) \mid<\delta/2$.
Consequently, $\phi$ satisfies the desired properties.
\end{proof}

\bigskip
We can now state and prove a version of the Stone-Weierstrass theorem for continuous interval-valued functions, gathering the information obtained above.

For every $x\in K$ let $\overline{f_x}\in C(K,\mathcal{K_{C}})$ be the constant function which takes the constant value $f(x)$.
%If we assume that $W$ contains the constant functions, then we can obtain the following improved version of Theorem \ref{main}:

\begin{theorem}\label{main}
Let $H$ be a subset of $C(K,\mathcal{K_{C}})$ that contains the constant functions. Assuming that $Conv(H)$ separates points, then $H$ is $D_{\infty}$-dense in $C(K, \mathcal{K_{C}})$.
\end{theorem}

\begin{proof}
%
%Given $f \in C(K,\mathcal{K_{C}})$ and $\varepsilon > 0$, we can take, for each $x \in K$, $g_x:=\widehat{f(x)}$.
%In this case $ d_H (f(x), g_x(x))=0$ and the result follows from Theorem \ref{main}.

Let $f$ be in  $C(K,\mathcal{K_{C}})$ and fix $\varepsilon > 0 $. We  need to find $g\in H$ such that $D_{\infty}(f,g)<\varepsilon$.

Take $x\in K$ and $0<\varepsilon(x)<\varepsilon$ and  let us define the following open neighborhood of $x$:
 $$V(x) := \{ t \in K: d_H(f(t), f(x)) < \varepsilon(x)<\varepsilon \}.$$
Apply, then,  Lemma~\ref{pat3} to get $U(x)$, an open neighborhood of $x$, satisfying the properties there.

Fix any point $x_1 \in K$ and take $W=K\setminus V(x_1)$ which turns out to be a compact set; so we can find 
a finite number of points, namely  $x_2, 	\ldots, x_m$,  in $W$ such that $$W \subset U(x_2) \cup \ldots \cup U(x_m).$$
Take $M= \max_{1\leq i\leq m} \{ D_{\infty}(f, \overline{f_{x_i}})\}$ and $\varepsilon'= \max_{ 1 \le i \le m} \{ \varepsilon(x_i) \}$ and
let us choose $0 < \delta < \frac{1}{2}$ such that $\delta M m < \varepsilon - \varepsilon'$.

On the other hand,  Lemma~\ref{pat3} also gives $\phi_2, \cdots, \phi_m \in Conv(H)$ such that, for all $i=2, \ldots, m$,
\begin{enumerate}
\item[(i)] $\phi_i (t) > 1 - \delta,$ for all $t \in U(x_i)$;
\item[(ii)] $0 \le \phi_i (t) < \delta,\quad if  \enspace  t \notin V(x_i)$.
\end{enumerate}

Let us define the following functions which  belong to $Conv(H)$ as well:
\newline
$\psi_2 := \phi_2$,
\newline
$\psi_3 := ( 1 - \phi_2) \phi_3$,
\newline
$\vdots$
\newline
$\psi_m := ( 1 - \phi_2)( 1 - \phi_3) \cdots ( 1 - \phi_{m-1}) \phi_{m}$.

\medskip
%From the properties of Proposition~\ref{multiplier}, it is apparent that $\psi_i \in Conv(H)$ for all $i=2,\ldots,m$.
Next we  can compute the sum $\psi_2+\ldots+\psi_j$, $j=2,\ldots,m $
and we are going to show, by induction, that, 

$$\psi_2+\ldots+\psi_j= 1 - ( 1 - \phi_2)( 1 - \phi_3) \cdots ( 1 - \phi_j),\quad j=2,\ldots,m .$$
 It is clear that $\psi_2+\psi_3$ follows the rule since
$$\psi_2 + \psi_3 = \phi_2  + (1 - \phi_2) \phi_3= 1-(1-\phi_2)\cdot(1-\phi_3).$$
%$$ = 1-(1 - \phi_3 - \phi_2 + \phi_2\phi_3);$$
%$$\phi_2 + \phi_3 - \phi_2\phi_3 = 1 - 1 + \phi_3 + \phi_2 - \phi_2\phi_3 .$$
 Assume that it is also true for a certain $j\in\{4,...,m-1\}$ and let us check
$$\psi_2+\ldots+\psi_j+\psi_{j+1}= 1 - ( 1 - \phi_2)( 1 - \phi_3) \cdots ( 1 - \phi_j) ( 1 - \phi_{j+1}).$$
Then, it is easily checked that
\[\begin{split}
\psi_2+\ldots+\psi_j+\psi_{j+1}& =1 - ( 1 - \phi_2)( 1 - \phi_3) \cdots ( 1 - \phi_j)\\
&+( 1 - \phi_2)( 1 - \phi_3) \cdots ( 1 - \phi_{j}) \phi_{j+1}\\[1ex]
&=1 - ( 1 - \phi_2)( 1 - \phi_3) \cdots ( 1 - \phi_j) ( 1 - \phi_{j+1})
\end{split}\]
%$$\phi_2  + \ldots + ( 1 - \phi_2)( 1 - \phi_3) \cdots ( 1 - \phi_{j-1}) \phi_{j}+ ( 1 - \phi_2)( 1 - \phi_3) \cdots ( 1 - \phi_{j}) \phi_{j+1}=$$
%$$= 1-(1 - \phi_3 - \phi_2 + \phi_2\phi_3)\cdots (1 - \phi_{j+1} - \phi_j + \phi_j\phi_{j+1});$$
%$$\phi_2 + \phi_3 - \phi_2\phi_3 + \cdots +  \phi_{j+1} + \phi_j - \phi_j\phi_{j+1}   = 1 - 1 + \phi_3 + \phi_2 - \phi_2\phi_3 + \cdots + \phi_{j+1} + \phi_j - \phi_j\phi_{j+1}, $$
as was to be checked and letting us define $\psi_1 := (1 - \phi_2) \cdots (1 - \phi_m)$, which also belongs to $Conv(H) $ by Proposition~\ref{multiplier}, allows us 
to get $m$ functions in $Conv(H)$ satisfying 
 $$\psi_1 + \psi_2+ \ldots + \psi_m = 1.$$
Next we claim that
\begin{equation}\label{1}
\psi_i(t) < \delta \hspace{0.1in} {\rm for} \hspace{0.04in} {\rm all} \hspace{0.04in} t \notin V(x_i), \hspace{0.04in}i = 1, \ldots, m.
\end{equation}
It is apparent when $i \ge 2$, since $\psi_i(t) \le \phi_i(t) < \delta$ for all $t \notin V(x_i)$ from (ii) above. So we only need to show it for $i=1$.
If $t \notin V(x_1)$, we have $t \in W$.
Hence, $t \in U(x_j)$ for some $j=2,\ldots,m$. From (i), we have $1 - \phi_j(t) < \delta $ and then
$$\psi_1(t) = (1 - \phi_j(t)) \prod_{i \not= j} ( 1 - \phi_i(t)) < \delta.$$
Finally, we can define
\begin{equation}\label{g}
g := \psi_1 \overline{f_{x_1}} + \psi_2\overline{f_{x_2}} + \ldots + \psi_m \overline{f_{x_m}},
\end{equation}
%where $g_i$ stands for $g_{x_i}$ for $i=1,\ldots,m$. A routine manipulation yields
%$$g= \phi_2 g_2 + ( 1 - \phi_2 ) [\phi_3 g_3 + (1 - \phi_3) [\phi_4 g_4 + \cdots + (1 - \phi_{m-1}) [\phi_m g_m + (1 - \phi_m) g_1] \cdots]].$$
Since $\psi_i\in Conv(H)$ for $i=1,...,m$ (see Definition \ref{multi}), we have  $g \in H$ 
and it remains to show that this function $g$ satisfies the requested property. 

\medskip
Fix now $x_0 \in K$ and, from Proposition \ref{properties}, we get that
\[
\begin{split}
d_H (f(x_0), g(x_0)) &= d_H \left(\sum^m_{i=1} \psi_i (x_0)f(x_0), \sum^m_{i=1} \psi_i (x_0)f(x_i))\right) \\
&\le \sum^m_{i=1} \psi_i (x_0) d_H (f(x_0), f(x_i)).
\end{split}
\]
Let us split the set $\{1,2\ldots,m\}$ into two disjoint sets: $\mathcal{I}=\{ 1 \le i \le m: x_0 \in V(x_i) \}$ and $\mathcal{J}=\{ 1 \le i \le m: x_0 \notin V(x_i) \}$. Then, for all $i \in \mathcal{I}$, we have
$$\psi_i (x_0) d_H (f(x_0), f(x_i)) \le \psi_i(x_0) \varepsilon'$$
and, for all $i \in \mathcal{J}$, the inequality (\ref{1}) yields
$$\psi_i(x_0) d_H (f(x_0),f(x_i)) \le \delta M.$$
From these two inequalities,  we deduce
$$\sum^m_{i=1} \psi_i (x_0) d_H (f(x_0), f(x_i)) \le \sum_{i \in \mathcal{I}} \psi_i(x_0)\varepsilon' + \sum_{i \in \mathcal{J}} \delta Q \le \varepsilon' + \delta Mm < \varepsilon. $$
Finally, gathering all the information above, we infer $d_H (f(x_0), g(x_0)) < \varepsilon$ for all $x_0\in K$, which yields $D_{\infty}(f,g) \le \varepsilon$ as desired.

\end{proof}

Given $J\in \mathcal{K_{C}}$, we shall keep writing $\overline{J}$ to denote the function in $C(K, \mathcal{K_{C}})$ which takes the constant value $J$.
\begin{corollary} \label{demoredes} Given $f \in C(K, \mathcal{K_{C}})$, there exist finitely many functions $ \psi_i \in C(K, [0, 1])$ and $J_i\in \mathcal{K_{C}}$, $i=1,...,m$, such that
$$ D_{\infty}(f, \psi_1 \overline{J_1} + ... + \psi_m \overline{J_m} ) < \varepsilon.$$
\end{corollary}

\begin{proof}
Let us first remark that $Conv(C(K, \mathcal{K_{C}}))=C(K, [0,1])$, which clearly separates the points of $K$. Then, it is enough to take $J_i:=f(x_i)$, $i=1,...,m$, in the definition  of the function $g$ in the proof of Theorem \ref{main} (see formula~(\ref{g})).

\end{proof}

\section{A Jackson type approximation result for interval-valued functions.}

\bigskip
As pointed out in the corollary above, the family of all finite sums of the form $\sum_{i=0}^{m} \psi_i(x)A_i$, where $\psi_i(x)$ are continuous functions from $K$ into $[0,1]$ and $A_i$ are closed intervals of the real line, is dense in $C(K,\mathcal{K_C})$. When we consider the particular case $K=[a,b]$, we will be able to provide an upper bound  of the approximation error  between a continuous function and a member of such family, obtaining a Jackson-type result.  Similar results can be in found in, e.g., \cite{Chen}, \cite{Hong-Hahm:2002} and \cite{Hong-Hahm:2016} for classical neural networks and in \cite{gal:94} in the fuzzy setting. We follow the techniques in \cite{Chen}, but some difficulties arise because we deal with interval arithmetic, whose properties differ considerably  from those of the arithmetic of real numbers.
%In the proof below we are using a similar technique from \cite{Chen} (see also \cite{Hong-Hahm:2016}).
Without loss of generality, we will consider the unit interval instead of $[a,b]$.

\medskip
Let $f\in C([0,1],\mathcal{K_C})$. The modulus of continuity of $f$ is defined to be

\[
\omega(f,\delta):=\sup\{d_H(f(x),f(y))\ : \ x,y\in[0,1]; |x-y|<\delta\}.
\]

Let  $\mathfrak{T}_n$ denote the family
$$\left\{\sum_{i=0}^{n} \psi_i(x)A_i\ :\ \psi_i \in C([0,1],[0,1]); \ A_i\in\mathcal{K_C},=0,1,2,\ldots n\right\}.$$
Define the  approximation error between a member of $\mathfrak{T}_n$ and a continuous function $f\in C([0,1],\mathcal{K_C})$ by
$$
E_{n,f}:=\inf_{g\in \mathfrak{T}_n}D_{\infty}(f,g)
$$

As usual, the main concern when dealing with interval-valued functions is to find a well-behaved substraction for intervals with respect to the Hausdorff metric and with some sort of cancelation law. Namely,  Minkowski difference has not the desired properties but there have been other approaches to provide suitable interval differences (see, for instance, Hukuhara \cite{huku:67}, Markov \cite{markov:77},
Lodwick \cite{lodwick:99}, Chalco-Cano et al \cite{CLB:2014}). Maybe the most used, due to its simplicity, is the  Hukuhara difference which was generalized in 2008 by L. Stefanini (\cite{Stefanini:2008},\cite{Stefanini:2010}). This generalization has the requested properties we need.

Let us denote by $A\circleddash B$ the generalized Hukuhara difference, $gH$-difference for short, defined in  \cite{Stefanini:2008}  by
\[
A\circleddash B =C \Leftrightarrow \begin{cases} A=B+C\\ \text{or}\\ B=A+(-1)C \end{cases}
\]
 It is worth remarking that the $gH$-difference always exists in $\mathcal{K_C}$ and
in fact,
\[
[a^-,a^+]\circleddash [b^-,b^+]=[\min\{a^--b^-,a^+-b^+ \},\max\{a^--b^-,a^+-b^+\}]
\]
If $A=\{a\}$ and $B=\{b\}$ are two singletons, then
$A\circleddash B= a-b$.

We denote by $L(A)$ the diameter or length of the interval $A$; that is, $L(A)=a^+-a^-$, when $A=[a^-,a^+]$.

The following properties will be useful in the sequel and can be found in \cite{Stefanini:2010}:

\begin{proposition}\label{gH}
Let $A=[a^-,a^+]$, $B=[b^-,b^+]$  be in $\mathcal{K_C}$. Then
\begin{enumerate}
\item $A\circleddash A={0}$.
\item Either
 $A+(B\circleddash A)=B$ or $B-(B\circleddash A)=A$.
\item If $L(A)\leq L(B)$, then $A+(B\circleddash A)=B$.
\item $d_H(A\circleddash B, 0)=d_H(A, B)$
\end{enumerate}
\end{proposition}
The cancelation law (item (3) above) allows us to provide the following property:
\begin{proposition}\label{propo:cancel}
Let $\{A_j\}_{j=0}^{n}$ be a family of subsets of $\mathcal{K_C}$.
\begin{enumerate}
\item If
$L(A_j)\geq L(A_{j+1})$, for all $j=0,1,\ldots, n-1$
then,
for any $0\leq k\leq n-1$,
\[
A_{n}+\sum_{j=k}^{n-1}\left(A_{j}\circleddash A_{j+1}\right)=A_{k}
\]
\item
If
$L(A_j)\leq L(A_{j+1})$, for all $j=0,1,\ldots, n-1$
then,
for any $0\leq k\leq n-1$,
\[
A_{0}+\sum_{j=0}^{k}\left(A_{j+1}\circleddash A_{j}\right)=A_{k+1}
\]
\end{enumerate}
\end{proposition}
\begin{proof}
For (1) we use that, for every $j$, $A_{j+1}+(A_j\circleddash A_{j+1})=A_{j}$.
\[
\begin{split}
 A_{n}+\sum_{j=k+1}^{n-1}\left(A_{j}\circleddash A_{j+1}\right)&= A_{n}+\left(A_{n-1}\circleddash A_{n}\right)\\
&+\left(A_{n-2}\circleddash A_{n-1}\right) + \ldots\\[2ex]
&=A_{n-1} + \left(A_{n-2}\circleddash A_{n-1}\right) \\
&+ \left(A_{n-3}\circleddash A_{n-2}\right)+ \ldots\\[2ex]
&\vdots\\
&=A_{k+2} +\left(A_{k+1}\circleddash A_{k+2}\right)+\left(A_{k}\circleddash A_{k+1}\right)\\
&=A_{k+1} +\left(A_{k}\circleddash A_{k+1}\right) =A_{k}\\
\end{split}
\]
The proof of (2) is  similar but using $A_{j}+(A_{j+1}\circleddash A_{j})=A_{j+1}$, for every $j$.
\end{proof}

\medskip
Next set $A_j:=f(j/n)\in\mathcal{K_C}$ for  $j=0,1,\ldots,n-1$, where $f\in C([0,1],\mathcal{K_C})$.
Then, Proposition~\ref{propo:cancel} forces us to demand some kind of
 monotonicity as follows: for a continuous function $f\in C([0,1],\mathcal{K_C})$ 
 we define the length function $len(f) : [0,1]\rightarrow \mathbb{R}$ as 
 \[
 len(f)(x)=f(x)^+ - f(x)^-,\quad \text{ for all $x \in[0,1]$}
\]
where $f(x)=[f(x)^-,\, f(x)^+]$.
 As usual, the length function is said non-increasing when 
\begin{equation}\label{eq:length1}
len(f)(x)\geq len(f)(y),\quad \text{when $x\leq y$.}
\end{equation}
and non-decreasing when
\begin{equation}\label{eq:length2}
len(f)(x)\leq len(f)(y),\quad \text{when $x\leq y$.}
\end{equation}

For instance, type (i)-gH-differentiable functions  satisfy \eqref{eq:length2} and 
type (ii)-gH-differentiable functions  satisfy \eqref{eq:length1} (see \cite[Proposition 3.8]{Chalco:2021}).

So, for every $j=0,1,\ldots n-1$, we have:
$$A_{j+1}+(A_j\circleddash A_{j+1})=A_{j},\quad  \text{when~\eqref{eq:length1} is satisfied.}$$
and
$$A_{j}+(A_{j+1}\circleddash A_{j})=A_{j+1},\quad \text{when~\eqref{eq:length2} is satisfied.}$$

%\[
%A_{j}\pm\left(A_{j(-)}^{j+1(+)}\circleddash A_{j+1(-)}^{j(+)}\right)=A_{j+1}
%\]
%where in the case $A_{j}+\ldots=A_{j+1}$ we must follow the (+) path and in the case $A_{j}-\ldots=A_{j+1}$
%we must follow the (-) path.
Now, we are ready to show the following Jackson-type theorem:

\begin{theorem}\label{th:degree1}
Let $f\in C([0,1],\mathcal{K_C})$ satisfying (\ref{eq:length1}). Then, for every $n\in\mathbb{N}$,

\[E_{n,f} \leq 2\omega \left(f,\dfrac{1}{n}\right)\]

\end{theorem}

\begin{proof}
Fix $\varepsilon>0$ and $n\in \mathbb{N}$. Consider the points $a_j=\frac{j}{n}$, $j=0,1,\ldots,n$. Take $0<\delta<\frac{1}{2n}$.
Being $f$ a continuous function on a compact set, it is bounded, so we can find  a constant $M>0$ such that $d_H(f(x),f(y))\leq 2M$, for all $x,y\in[0,1]$.
Take $\epsilon'=\frac{\varepsilon}{2nM}$ and apply Lemma \ref{patata} to $a_j-\delta<a_j+\delta$. So we can find $n+1$ continuous functions
$\phi_j$,  $j=0,1,\ldots n$, with
\begin{enumerate}
\item $\phi_j(t)<\epsilon'$, for $t\geq a_j+\delta$, for all $j=0,1,\ldots, n-1$;
\item $\phi_j(t)>1-\epsilon'$, for $t\leq a_j-\delta$, for all $j=1,\ldots, n$.
\end{enumerate}
%It is worth to remark that $\phi_0(t)<\epsilon'$, for $t\geq \delta$ and $\phi_n(t)>1-\epsilon'$, for $t\leq 1-\delta$.

%Now we define the following functions, modifying slightly the above functions to get a set with sum equal to one, but preserving the good properties they have.
%\begin{eqnarray*}
%\psi_1&=&\phi_1(1-\phi_2)\cdots (1-\phi_n), \\
%\psi_2&=&(1-\phi_1)\phi_2(1-\phi_3)\cdots (1-\phi_n),\\
%&\vdots&\\
%\psi_n&=&(1-\phi_1)(1-\phi_2)\cdots (1-\phi_{n-1})\phi_n
%\end{eqnarray*}
%and
%\begin{eqnarray*}
%\psi_0&=&1-\sum_{j=1}^{n} \psi_j.
%\end{eqnarray*}
%For $t\geq\delta$, there exists $1\leq k\leq n$ with $t\leq a_k-\delta$, so $\phi_k(t)>1-\epsilon'$ or $t> a_{n-1}+\delta$ and then $\phi_{n-1}(t)<\epsilon'$. Anyway, in the first case,
%$$
%\psi_0(t) =1-\sum_{j=1}^{n} \psi_j(t)\leq 1-\psi_k(t)<\epsilon'
%$$
%
%That means that the set of functions $\{\psi_j(t)\}_{j=0}^{n}$ satisfies the same properties of $\{\phi_j\}_{j=0}^{n}$ and $\sum_{j=0}^n \psi_j(t)=1$, for all $t\in [0,1]$.

For every $j=0,1,\ldots,n$, we denote by  $\psi_j$ the restriction of $\phi_j$ to the interval $[0,1]$ and set
$A_j=f(a_j)$. Define now

$$g(x):=A_{n}+\sum_{j=0}^{n-1} \psi_j (x)\Big(A_{j}\circleddash A_{j+1}\Big)\in\mathfrak{T}_n$$

Now, take any $x\in[0,1]$. We are going to study the value of $d_H(f(x),g(x))$ and two cases arise depending on where $x$  is located.

\medskip
\textsc{Case 1.} There exists $k$ with $x\in [a_k-\delta,a_k+\delta]$ (or $0\leq x\leq\delta$ if $k=0$ or $1-\delta\leq x\leq 1$ if $k=n$).

In this case, $x>a_j+\delta$, for $j<k$ and $x<a_j-\delta$, for $j>k$. Then
$\psi_j(x)<\epsilon'$ for $j<k$ and $1-\psi_j(x)<\epsilon'$, for $j>k$.

Then we can rewrite $g(x)$ as
\begin{equation}\label{eq:rewriteg1}
\begin{split}
g(x)&=\sum_{j=0}^{k-1} \psi_j (x)\Big(A_{j}\circleddash A_{j+1}\Big)\\
&+ \psi_k(x)\Big(A_{k}\circleddash A_{k+1}\Big)\\
&+\sum_{j=k+1}^{n-1} \psi_j (x)\Big(A_{j}\circleddash A_{j+1}\Big)+A_{n}\\
\end{split}
\end{equation}
Using (5) from Proposition \ref{properties} and (1) from Proposition~\ref{propo:cancel}, we have
\[
\begin{split}
d_H(g(x),f(x))&=d_H(g(x)+A_{k},\ f(x)+A_{k})\\
&=d_H\Big(g(x)+A_{k},\ f(x)+A_{n}+(A_{k}\circleddash A_{k+1})+\sum_{j=k}^{n-1}\left(A_{j}\circleddash A_{j+1}\right) \Big)\\
\end{split}
\]

Finally, by (1) and (2) in Proposition~\ref{properties}, we infer
\[
\begin{split}
d_H(g(x),f(x))&\leq \sum_{j=0}^{k-1}  \psi_j (x) d_H(A_{j}\circleddash  A_{j+1},\ \mathbf{0}) \\
&+  d_H\Big(\psi_k(x)(A_{k}\circleddash  A_{k+1}),\ A_{k}\circleddash  A_{k+1}\Big)\\
&+\sum_{j=k+1}^{n-1} d_H\Big(\psi_j(x)(A_{j}\circleddash  A_{j+1}),\ A_{j}\circleddash  A_{j+1}\Big)\\
&+d_H(A_k,\ f(x))\\
&\leq \sum_{j=0}^{k-1}  \psi_j (x) d_H(A_{j},\  A_{j+1}) \\
&+  (1-\psi_k(x))d_H(A_{k},\   A_{k+1})\\
&+\sum_{j=k+1}^{n-1} (1-\psi_j(x))d_H(A_{j},\  A_{j+1})\\
&+d_H(A_k,\ f(x))\\
 &\leq k\epsilon'\; 2M + \omega(f,1/n) + (n-1-k)\epsilon'\; 2M +\omega(f,1/n)\\
 &< 2\omega(f,1/n)+\varepsilon
\end{split}
\]

\textsc{Case 2.} There exists $0\leq k\leq n-1$ with $x\in (a_k+\delta, a_{k+1}-\delta)$.
In this case, $x>a_j+\delta$, for $j\leq k$ and $x<a_j-\delta$, for $j>k$. Then
$\psi_j(x)<\epsilon'$ for $j\leq k$ and $1-\psi_j(x)<\epsilon'$, for $j>k$.

Now, we can rewrite $g(x)$ as
\begin{equation}\label{eq:rewriteg2}
\begin{split}
g(x)&=\sum_{j=0}^{k} \psi_j (x)\Big(A_{j}\circleddash A_{j+1}\Big)\\
%&+ \psi_k(x)\Big(A_{k}\circleddash A_{k+1}\Big)\\
&+\sum_{j=k+1}^{n-1} \psi_j (x)\Big(A_{j}\circleddash A_{j+1}\Big)+A_{n}\\
\end{split}
\end{equation}

And, as in the previous case, but bearing in mind that now $d_H(A_{k+1},f(x))\leq \omega(f,1/n)$ as well,
\[
\begin{split}
d_H(g(x),f(x))&=d_H(g(x)+A_{k+1}, \ f(x)+A_{k+1})\\
&=d_H\Big(g(x)+A_{k+1},\  f(x)+A_n+\sum_{j=k+1}^{n-1} (A_j\circleddash A_{j+1})\Big)\\
&\leq  \sum_{j=0}^{k}  \psi_j (x) d_H(A_{j},  A_{j+1}) \\
&+\sum_{j=k+1}^{n-1}(1- \psi_j (x)) \, d_H(A_{j},A_{j+1})\\
 &+d_H(A_{k+1},f(x))\\
 &\leq (k+1)\epsilon'\; 2M + (n-1-k)\epsilon'\; 2M +\omega(f,1/n)\\
 &< 2\omega(f,1/n)+\varepsilon
\end{split}
\]

So we get, for all $x\in[0,1]$,
$$
d_H(g(x),f(x))\leq 2\omega\left(f,\frac{1}{n}\right) +\varepsilon
$$
which means, $$D_{\infty}(f,g)=\sup_{x\in[0,1]} d_H(f(x),g(x))\leq 2\omega\left(f,\frac{1}{n}\right) +\varepsilon$$

%\medskip
%Bearing in mind that, since $\sum_{j=0}^n\psi_j(x)=1$, the constant functions are members of $\mathfrak{T}_n$ as well, we can write
%$$g(x)=\sum_{j=0}^n\psi_j(x) \Big[ f(a_{n+1})+f(a_j)-f(a_{j+1})\Big]$$
%which belongs to $\mathfrak{T}_n$.

Then,

$$\inf_{h\in \mathfrak{T}_n}D_{\infty}(f,h)\leq D_{\infty}(f,g)\leq 2\omega\left(f,\frac{1}{n}\right) +\varepsilon$$
and being true for all $\varepsilon>0$, we finally get
$$E_{n,f} \leq 2\omega\left(f,\frac{1}{n}\right). $$

\end{proof}

\begin{theorem}\label{th:degree2}
Let $f\in C([0,1],\mathcal{K_C})$ satisfying \ref{eq:length2}. Then, for every $n\in\mathbb{N}$,
$$
E_{n,f} \leq 2\omega \left(f,\dfrac{1}{n} \right)
$$
\end{theorem}
\begin{proof}
The proof goes parallel to the previous theorem but taking $\psi_i=1-\phi_i$ and considering the function  in $\mathfrak{T}_n$ defined by
\[
g(x):=A_0+\sum_{j=0}^{n-1}\psi_j(x)(A_{j+1}\circleddash A_{j})\in\mathfrak{T}_n
\]
\end{proof}

\section{Approximation by means of interval-valued neural networks}

In \cite{ishibuchi-tanaka:93} (see also \cite{patil:95}), an architecture of neural networks with interval weights and interval biases was proposed. Such networks maps an input  of real numbers to an output interval and they are called interval neural networks.

Based on the results of the previous section and taking advantage of an striking result by Guliyev and Ismailov (\cite{GI}), we show how an interval-valued continuous functions can be approximated using interval neural networks. Namely, we shall deal with interval neural networks of the following form:
$$H(x) = \sum^m_{i=1} W_i \left(c_{i1} \cdot \sigma ( x - \theta_{i1} )+c_{i2} \cdot \sigma ( x - \theta_{i2} )\right)$$
for each $x\in \mathbb{R}$, where  the weights $W_i \in \mathcal{K_C}$ and the weights $c_{ij}$ and the thresholds $\theta_{ij}$ are real numbers for $i=1,...m$ and $j=1,2$. Here $\sigma: \mathbb{R} \to \mathbb{R}$ stands for the activation function in the hidden layer.
%Given an activation function $\sigma$, we denote the set of all four-layer INNs by $\mathcal{H}(\sigma)$.

%We will show, based on the results in the previous section, that the elements of $\mathcal{H}(\sigma)$ can approximate any continuous interval-valued function defined on a compact subspace of $\mathbb{R}^n$ provided the activation function
%$\sigma$ be a non-polynomial continuous function.
%%

\begin{theorem}
Let $K$ be a compact subset of $\mathbb{R}$ and let $f \in C(K, \mathcal{K_{C}})$ and $\varepsilon>0$.
Then there exists a sigmoidal function $\sigma: \mathbb{R} \to \mathbb{R}$ and an interval neural network $H(x)$ as above such that $D_{\infty}(f,H) < \varepsilon.$
\end{theorem}

\begin{proof}
By Corollary \ref{demoredes}, there exist finitely many functions $ \psi_i \in C(K, [0, 1])$ and $J_i\in \mathcal{K_{C}}$, $i=1,...,m$, such that
$$ D_{\infty}(f, \psi_1 \widehat{J_1} + ... + \psi_m \widehat{J_m} ) < \frac{\varepsilon}{2}.$$

On the other hand, by \cite[Theorem 4.2]{GI}, we know that, for each $\psi_i$, $i=1,...,m$, there exist weights $c_{i1}$, $c_{i2}$, thresholds $\theta_{i1}$, $\theta_{i2}$, in $\mathbb{R}$ and a sigmoidal function $\sigma: \mathbb{R} \to \mathbb{R}$ such that

$$\left| \psi_i(x) - \sum^m_{i=1} \left(c_{i1} \cdot \sigma ( x - \theta_{i1} )+c_{i2} \cdot \sigma ( x - \theta_{i2} )\right) \right|< \frac{\varepsilon}{2m\cdot d_H(J_i,0)},$$
for all $x \in K$. Hence
\[
\begin{split}
&D_{\infty}\left(\left(c_{i1} \cdot \sigma ( x - \theta_{i1} )+c_{i2} \cdot \sigma ( x - \theta_{i2} )\right)\widehat{J_i}, \psi_i(x) \cdot \widehat{J_i}\right)\\
&=\sup_{x \in K} d_H \left(\left(c_{i1} \cdot \sigma ( x - \theta_{i1} )+c_{i2} \cdot \sigma ( x - \theta_{i2} )\right)J_i,\psi_i(x) \cdot J_i\right) \\
&= \sup_{x \in K} \left| \psi_i(x) - \left(c_{i1} \cdot \sigma ( x - \theta_{i1} )+c_{i2} \cdot \sigma ( x - \theta_{i2} )\right) \right| d_H(J_i,0)<\frac{\varepsilon}{2m}
\end{split}
\]
%by property (3) of the distance $d_H$
%$$= \sup_{x \in K} \left| \psi_i(x) - \left( \sum^s_{j=1} w_{ij} \cdot \sigma ( x \cdot a_j + \theta_j )\right)  \right|  d_\infty(u_i,0)<\frac{\varepsilon}{2m},$$
which clearly yields
$$ D_{\infty}\left(f, \hspace{0.05in}\sum^m_{i=1} \widehat{J_i} \left(c_{i1} \cdot \sigma ( x - \theta_{i1} )+c_{i2} \cdot \sigma ( x - \theta_{i2} )\right)\right) < \varepsilon.$$

\end{proof}

 \section{Conclusion}
We have proved that, under certain natural assumptions and using only straightforward concepts, we can uniformly approximate, to any degree of accuracy,  any continuous function from a compact Hausdorff space into the space of closed intervals, which yields a Stone-Weierstrass type result in this setting.
%We have also shown an application of these results to the context of set-valued dynamical systems by providing a numerical  example of how to get approximative trajectories to any fixed exact trajectory.
%As a corollary, we infer that fuzzy-number-valued four-layer regular neural networks can approximate any continuous fuzzy-number-valued function defined on a compact subspace of $\mathbb{R}^n$.
%
%\section*{Acknowledgment}
%This research is supported by Spanish Government (MTM2016-77143-P),
%Universitat Jaume I  (P1-1B2014-35) and Generalitat Valenciana (Projecte AICO/2016/030).


\begin{thebibliography}{widest-label}

\bibitem{aubin-franskowska:1990} J. P. Aubin, H. Franskowska, \textit{Set Valued Analysis}, Birkhauser, London (1990).

%\bibitem{baker-patil:95} M. R. Baker, R. B. Patil,
%\textit{Universal approximation theorem for interval neural networks}, Reliable Computing, {\bf 4} (1998),
%235--239.
%
%\bibitem{BJ} H.T. Banks, M.Q. Jacobs, \textit{ A differential calculus for multi-functions}, J. Math. Anal. Appl., \textbf{29}, (1970) 246--272.

%\bibitem{BG} B. Bede, S. Gal,
%\textit{Generalizations of the differentiability of fuzzy-number-valued functions with applications to fuzzy differential equations}, Fuzzy Sets and Systems, {\bf 151} (2005),
%581--599.


\bibitem{DK2} P. Diamond, P. Kloeden,
\textit{Metric Spaces of Fuzzy Sets: Theory and Applications}, World Scientific, Singapore, (1994).

\bibitem{dwyer} P. S. Dwyer, \textit{ Linear Computations}, John Wiley and Sons, New York, (1951).

%\bibitem{DP} D. Dubois, H. Prade,
%\textit{Operations on fuzzy numbers}, Internat. J. of Systems Sci., {\bf 9} (1978), 613--626.
%

%\bibitem{FSS18}
%J. J. Font, D. Sanchis, M. Sanchis, \textit{Completeness, metrizability and compactness in spaces of fuzzy-number-valued functions}
%Fuzzy Sets Syst, {\bf 353} (2018), 124--136.


\bibitem{FSS:17}
J. J. Font, D. Sanchis, M. Sanchis, \textit{A version of the Stone-Weierstrass theorem in fuzzy analysis}
J. Nonlinear Sci. Appl., 10 (2017), 4275--4283.

%\bibitem{FX} J-X. Fang and Q-Y. Xue,
%\textit{Some properties of the space of fuzzy-valued continuous functions on
%a compact set}, Fuzzy Sets and Systems, {\bf 160} (2009), 1620--1631.


\bibitem{GI} N.J. Guliyev, V.E. Ismailov,
\textit{On the approximation by single hidden layer feedforward neural
networks with fixed weights}, Neural Networks, {\bf 98} (2018),
296--304.

%\bibitem{he} E. Helly,
%Über Systeme linearer Gleichungen mit unendlich vielen Unbekannten,
%Monatsh. Math. Phys. {\bf 31} (1921) no. 1, 60--91. (in German)


%\bibitem{HGL} S. Hai, Z. Gong, H. Li,
%\textit{Generalized differentiability for n-dimensional fuzzy-number-valued functions and fuzzy optimization}, Inf. Sci., {\bf 374} (2016),
%151--163.


%\bibitem{HW} H. Huang, C. Wu,
%\textit{Approximation of level fuzzy-valued functions by multilayer regular fuzzy neural networks}, Math. Comp. Model., {\bf 49} (2009),
%1311--1318.

%\bibitem{hukuhara} M. Hukuhara, \textit{Integration des applications mesurables dont la valeur est un compact convexe}, Funkc. Ekvacioj, \textbf{10},
%(1967),  205--223.

\bibitem{ishibuchi-tanaka:93} H. Ishibuchi, H. Tanaka, H. Okada,
\textit{An architecture of neural networks with interval weights and its application to fuzzy regresion analysis}, Fuzzy Sets and Systems, {\bf 57} (1993),
27-39.



\bibitem{Je} R.I. Jewett,
\textit{A variation on the Stone-Weierstrass theorem}, Proc. Amer. Math. Soc., {\bf 14} (1963),
690--693.




%\bibitem{LLPS} M. Leshno, V.Y. Lin, A. Pinkus, S. Schocken, \textit{Multilayer feedforward networks with a nonpolynomial activation function can approximate any function}, Neural Networks, {\bf 6 } (1993), 861--867.
%
%
%\bibitem{LP} V.Y. Lin and A. Pinkus, \textit{Fundamentality of ridge functions}, J. Approx. Th., {\bf 99} (1999), 68--94.
%
%
%\bibitem{Liu} P.Y. Liu,
%\textit{Universal approximations of continuous fuzzy-valued
%function by multi-layer regular fuzzy neural networks}, Fuzzy Sets and Systems, {\bf 119} (2001),
%313--320.

%\bibitem{macario}
%S. Macario, M. Sanchis, \textit{Gromov-Hausdorff convergence of non-Archimedean fuzzy metric spaces}
%Fuzzy Sets and Systems, {\bf 267} (2015), 62--85.



\bibitem{moore1} R. E. Moore, \textit{ Automatic error analysis in digital computation}, LMSD Technical report
48421, Lockheed Aircraft Corporation, Sunnyvale, California, January 1959.
%Also at http://interval.louisiana.edu/Moores_early_papers/bibliography.html.

\bibitem{moore2} R. E. Moore, C. T. Yang, \textit{ Interval Analysis I}, LMSD Technical report 285875, Lockheed
Aircraft Corporation, Sunnyvale, California, September 1959.
%Also at http://interval.louisiana.edu/Moores_early_papers/bibliography.html.

\bibitem{moore:66}  R. E. Moore, Interval Analysis. Prentice-Hall, Englewood Cliffs N. J., 1966.

\bibitem{patil:95} R. B. Patil,
\textit{Interval neural networks} in \textit{Extended Abstracts of APIC'95: International Workshop on Applications of Interval Computations, El Paso, TX, 1995}, Reliable Computing, (1995). Supplement.

%\bibitem{pilyugin-rieger:2008} S. Y. Pilyugin, J. Rieger
%\textit{Shadowing and inverse shadowing in set-valued dynamical systems. Contractive case.},
%Top. Methods in Nonlin. Analysis, Journal of the Juliusz Schauder Center, {\bf 32} (2008), 139--149.

\bibitem{prola} J.B. Prolla,
\textit{On the Weierstrass-Stone Theorem},
J. Approx. Theory, {\bf 78} (1994), 299--313.


\bibitem{sunaga} Sunaga, T., \textit{Theory of interval algebra and its application to numerical analysis}, RAAG
Memoirs, Ggujutsu Bunken Fukuy-kai, Tokyo, 2 (1958), 29-48.
%Also at http://www.cs.utep.edu/interval-comp/sunaga.pdf.


\bibitem{warmus} M. Warmus,
\textit{Calculus of approximations}, Bulletin de l'Academie Polonaise de Sciences,
\textbf{4}, 5 (1956), 253-257.

%\bibitem{WG} C. Wu, Z. Gong,
%\textit{On Henstock integral of fuzzy-number-valued functions (I)}, Fuzzy Sets and Systems, {\bf 120} (2001),
%523--532.
%















%\bibitem{PR} Puri, M. L.; Ralescu, D. A., \textit{ Differentials for fuzzy functions}, J. Math. Anal. Appl., \textbf{91}, (1983),  552--558.

%\bibitem{diamond-kloeden:1994} P. Diamond, P. Kloeden \textit{Metric spaces of fuzzy sets. Theory and applications}, World Scientific Pub. Co. Pte. Ltd, Singapore, (1994).
%







\bibitem{CLB:2014}
Y. Chalco-Cano, W. A. Lodwick, B. Bede,
\textit{Single level constraint interval arithmetic},
Fuzzy Sets and Systems, \textbf{257}, (2014), 146--168.


\bibitem{Chalco:2021}
Y. Chalco-Cano,  T. M. Costa, H. Rom\'an-Flores, A. Rufi\'an-Lizana,
\textit{New properties of the switching points for the generalized Hukuhara differentiability and some results on calculus},
Fuzzy Sets and Systems, \textbf{404}, (2021), 62--74.


\bibitem{Chen} D. Chen,
\textit{Degree of approximation by superpositions of a sigmoidal function}, Approximation Theory and its Applications, {\bf 9} (1993), no. 3,
17--28.

\bibitem{gal:94}
S.G. Gal, \textit{Degree of approximation of fuzzy mappings by fuzzy polynomials}, J. Fuzzy Math. \textbf{2} (1994)  (4), 847--853.

\bibitem{Hong-Hahm:2002} B. I. Hong; N. Hahm,
\textit{Approximation Order to a Function in
$C(\mathbb{R})$ by Superposition of a Sigmoidal Function
}, Applied Mathematical Letters, {\bf 15} (2002), 591--597.

\bibitem{Hong-Hahm:2016} B. I. Hong; N. Hahm,
\textit{A note on neural network approximation with a sigmoidal function}, Applied Math. Sciences., {\bf 10} (2016), no. 42,
2075--2085.

\bibitem{huku:67}
M. Hukuhara,
\textit{Integration des applications mesurables dont la valeur est un compact convexe},
Funkc. Ekvacioj, \textbf{10} (1967), 205--223.

\bibitem{Je} R.I. Jewett,
\textit{A variation on the Stone-Weierstrass theorem}, Proc. Amer. Math. Soc., {\bf 14} (1963),
690--693.

\bibitem{lodwick:99}
W. A. Lodwick, \textit{Constrained interval arithmetic},
Technical Report, Center for Computational Mathematics,
University of Colorado at Denver, Denver, USA, 1999.

\bibitem{markov:77}
S. Markov, \textit{
Extended interval arithmetic},
C. R. Acad. Bulgare Sci., \textbf{30} (1977), 1239--1242.

\bibitem{Stefanini:2008}L. Stefanini; \textit{
A generalization of Hukuhara difference for interval and fuzzy arithmetic},
D. Dubois, M.A. Lubiano, H. Prade, M.A. Gil, P. Grzegorzewski, O. Hryniewicz (Eds.), Soft Methods for Handling Variability and Imprecision, Series on Advances in Soft Computing, vol. 48, Springer (2008)
An extended version is available online at the RePEc service: http://econpapers.repec.org/RAS/pst233.htm

\bibitem{Stefanini:2010}
L. Stefanini,
\textit{A generalization of Hukuhara difference and division for interval and fuzzy arithmetic},
Fuzzy Sets Syst., \textbf{161} (2010), 1564--1584.



\end{thebibliography}
\end{document}